\documentclass{amsart}
\usepackage[mathscr]{eucal}
\usepackage{graphicx}
\usepackage{amscd}
\usepackage{amsmath}
\usepackage{amsthm}
\usepackage{amsxtra}
\usepackage{calc} 
\usepackage{amsfonts}
\usepackage{amssymb}
\usepackage{pdfsync}
\usepackage[all]{xy}
\usepackage[colorlinks, bookmarks=true]{hyperref}

\newtheorem{theorem}{Theorem}
\theoremstyle{plain}

\newtheorem{corollary}[theorem]{Corollary}
\newtheorem{corollary-definition}[theorem]{Corollary-Definition}

\newtheorem{lemma}[theorem]{Lemma}

\newtheorem{proposition}[theorem]{Proposition}

\numberwithin{equation}{section}

\theoremstyle{definition}

\newtheorem{remark}[theorem]{Remark}

\newcommand{\Ker}{\mathrm{Ker}\,}
\newcommand{\Cok}{\mathrm{Coker}\,}

\newcommand{\Hom}{\mathrm{Hom}}
\newcommand{\Ext}{\mathrm{Ext}}
\newcommand{\Tor}{\mathrm{Tor}}
\newcommand{\Imr}{\mathrm{Im}\,}

\newcommand{\Tr}{\mathrm{Tr}\,}

\newcounter{hours}
\newcounter{minutes}

\begin{document}
\title{On direct summands of homological functors on length categories}

\author{Alex Martsinkovsky}
\address{Mathematics Department\\
Northeastern University\\
Boston, MA 02115, USA}
\email{alexmart@neu.edu}
\date{\today, \setcounter{hours}{\time/60} \setcounter{minutes}{\time-\value
{hours}*60} \thehours\,h\ \theminutes\,min}

\begin{abstract}
We show that direct summands of certain additive functors arising as bifunctors with a fixed argument in an abelian category are again of that form whenever the fixed argument has finite length or, more generally, satisfies the descending chain condition on images of nested endomorphisms. In particular,
this provides a positive answer to a conjecture of M. Auslander in the case of categories of finite modules over artin algebras. This implies that the covariant Ext functors are the only injectives in the category of defect-zero finitely presented functors on such categories.
\end{abstract}
\maketitle

\section{Introduction}
This note concerns the conjecture of M.~Auslander that a direct summand of a covariant $\Ext$-functor is again of that form. More precisely, suppose $\mathscr{A}$ is an abelian category with enough projectives, $A$ is an object of $\mathscr{A}$, and $F$ is a direct summand of the functor $\Ext^{1}(A,-)$. Then the question is whether $F$ is of the form $\Ext^{1}(B,-)$ for some object $B$ of $\mathscr{A}$. The motivation for this problem comes from Auslander's foundational work~\cite{Aus66} on coherent functors from abelian categories to the category of abelian groups. It is an immediate consequence of 
Yoneda's lemma and the left-exactness of the $\Hom$ functor that each coherent functor gives rise (non-uniquely) to an exact sequence $X \to Y \to Z \to 0$ in the original abelian category. Specializing to such exact sequences which are short exact, one is naturally led to consider the corresponding subcategory $\mathscr{A}_{0}$ of coherent functors. Thus the objects of this subcategory are coherent functors whose projective resolutions are images of short exact sequences in $\mathscr{A}$ under the Yoneda embedding. An immediate example of such a functor is $\Ext^{1}(A,-)$.

As Auslander showed~\cite[Prop. 4.3]{Aus66}, if the answer to the above question is in the affirmative for any $A \in \mathscr{A}$ and any $F$, then the functors of the form 
$\Ext^{1}(B,-)$ are the only injectives in $\mathscr{A}_{0}$. He also 
showed~\cite[Prop.~4.7]{Aus66} that the answer is positive assuming that $A$ is of finite projective dimension. P.~Freyd \cite{Fr66} showed the same in the case $\mathscr{A}$ has countable sums. In \cite{Aus69} Auslander undertook a more systematic study of this problem and gave a unifying proof of the above two results. In addition, he provided a detailed analysis in the case when $\mathscr{A}$ was a category of modules over a ring. Among other things, he showed that, in general, a direct summand of the functor 
$\Ext^{1}(A,-)$, where $A$ is a finitely generated module over a ring $R$, need not be of the form $\Ext^{1}(B,-)$ for some finitely generated $R$-module $B$, even when $R$ is noetherian.

In this note, we return to general abelian categories and show that the answer to the above question is still positive if $A$ is noetherian and satisfies the descending chain condition on images of nested endomorphisms. 

In fact, this result  is proved in greater generality -- it holds for any additive functor of the form $G(A,-)$ (resp., $G(-,A)$), provided natural endo-transformations of such a functor are induced from endomorphisms of its fixed argument, see Theorem~\ref{main} for details. 

A preliminary result, establishing the Fitting lemma for objects of finite length in an abelian category, is recalled in the next section. In the section following the proof of Theorem~\ref{main}, we apply this theorem to the functors covariant 
$\mathrm{Ext}$, covariant 
$\Hom$ modulo projectives, and $\mathrm{Tor}$.

\section{Preliminaries}

This section contains preliminary results, all of which are well-known. For the convenience of the reader, their proofs are included in the case of a general abelian category, since the author could not find them in the literature.

Let $\mathscr{A}$ be an abelian category and $A$ an object of $\mathscr{A}$. Recall that the intersection of two subobjects of $A$ always exists and can be defined as the pullback of the corresponding  monomorphisms \cite[Prop. 4.2.3]{B1}. Furthermore, the union of two subobjects of $A$ always exists and can be defined as the pushout over their intersection \cite[Prop. 1.7.4]{B2}. In particular, for any endomorphism
$f : A \to A$, the intersection $\Ker f^{n} \cap \Imr f^{n}$ and the union 
$\Ker f^{n} + \Imr f^{n}$ are defined as subobjects of $A$ for any natural~$n$. 

The commutative diagram
\[
\xymatrix
	{
	A \ar@{=}[d] \ar[r]^{f} & A \ar[d]^{f}\\
	A \ar[r]^{f^{2}} & A
	}
\] 
shows that $\Ker f$ is a subobject of $\Ker f^{2}$ and, likewise, $\Ker f^{n}$ is a subobject of $\Ker f^{n+1}$. It also shows that $f$ induces an epimorphism 
$\Imr f \to \Imr f^{2}$ and, likewise, an epimorphism $\Imr f^{n} \to \Imr f^{n+1}$.

The commutative diagram
\[
\xymatrix
	{
	A \ar[d]^{f} \ar[r]^{f^{2}} & A \ar@{=}[d]\\
	A \ar[r]^{f} & A
	}
\] 
shows that $\Imr f^{2}$ is a subobject of $\Imr f$ and, likewise, $\Imr f^{n+1}$ is a subobject of $\Imr f^{n}$.

The object $A$ is said to be of finite length if it has both the descending chain condition (DCC) and the ascending chain condition (ACC). This is equivalent to saying that there exists a finite chain of subobjects 
\[
0 =  A_{0} \subset A_{1} \subset \ldots \subset A_{s} = A
\]
whose successive quotients are simple objects. The next lemma is elementary and well-known, at least for modules. For completeness, we give a categorical proof.

\begin{lemma}\label{pushout}
 Let
\[
\xymatrix
	{
0  \ar[r] & A \ar[r]^{a} \ar[d]^{b} &  B  \ar[r]^{r} \ar[d]^{d}  & C  \ar[r]  \ar@{=}[d] & 0 \\
0  \ar[r] & D  \ar[r]^{c} &  E  \ar[r]^{t} & C  \ar[r] & 0 	
	}
\] 
be a commutative diagram in $\mathscr{A}$ with exact rows. Then $(c,d)$ is a pushout of~$(a,b)$. 
\end{lemma}
\begin{proof}
 Let $(x,y)$ be a pushout of $(a,b)$. We then have a commutative diagram
 \[
\xymatrix
	{
	A \ar[dd]^{b} \ar@{>->}[drr]_{a} & & & & \\
	& & B \ar[ld]_{y} \ar@{->>}[drr]^{r} \ar [dd]_<<<<<<{d} & & \\
	D \ar@{>->}[r]^{x} \ar@{=}[dd] &  Z \ar@{..>}[dr]_{u} \ar@{->>} '[r] [rrr]^{s} & & &C \\
	& & E \ar@{->>}[urr]^{t} & & \\
	D \ar@{>->}[urr]^{c} & & & &
	}
\] 
of solid arrows, where the horizontal and the two slanted sequences are short exact. Since $da = cb$, the universal property of pushouts yields a map $u : Z \to E$ with $uy = d$ and $ux = c$. We claim that $s = tu$. Since $ra = 0 = tcb$, there is a unique morphism $q : Z \to C$ such that $qy = r$ and $qx = tc$. But both~$s$ and~$tu$ have this property, and therefore $s = tu$. Clearly, $u$ is an isomorphism. It now follows that $(c,d)$ is also a  pushout.
\end{proof}

Now we recall the Fitting lemma in an abelian category. 

\begin{lemma}\label{Fitting}
In the above notation, let $f : A \to A$ be an endomorphism. 
\begin{enumerate}

\item If $A$ has the ACC, then $0 \simeq \Ker f^{n} \cap \Imr f^{n}$ for all $n$ large enough. If $f$ is an epimorphism, then it is an isomorphism.
\smallskip

\item If $A$ has the DCC, then the inclusion of $\Ker f^{n} + \Imr f^{n}$ in $A$ is an isomorphism for all $n$ large enough. If $f$ is a monomorphism, then it is an isomorphism.
\smallskip
\item (The Fitting Lemma) If $A$ is of finite length, then $A \simeq \Ker f^{n} \oplus \Imr f^{n}$ for all $n$ large enough. The morphism $f$ is nilpotent on  $\Ker f^{n}$ for all $n$ and is an isomorphism on $\Imr f^{n}$ for all $n$ large enough. Also, $f$ is an isomorphism whenever it is a monomorphism or an epimorphism.

\end{enumerate}

\end{lemma}

\begin{proof}

(1) Since $A$ has ACC, the commutative diagram with exact rows
 \[
\xymatrix
	{
0  \ar[r] & \Ker f^{n} \ar[r] \ar@{>->}[dd] &  A  \ar[rr]^{f^{n}} \ar@{=}[dd]  \ar[dr] & & A  \ar[r] \ar[dd]^{f} & \Cok f^{n} \ar[r] \ar[dd] & 0\\
&&& \Imr f^{n} \ar[ur] \ar@{->>} [dd] && \\
0  \ar[r] & \Ker f^{n+1}  \ar[r] &  A  \ar '[r] [rr]^{f^{n+1}} \ar[dr] & & A  \ar[r] & \Cok f^{n+1} \ar[r] &  0\\
&&& \Imr f^{n+1} \ar[ur] && 
	}
\]
shows that the leftmost vertical morphism is an isomorphism for all $n \geq p$ for 
some~$p$. By the snake lemma\footnote{For a proof of the snake lemma in a general  abelian category, see~\cite{Fay}.}, the vertical morphism in the middle is also an isomorphism. The right-hand side of the diagram shows that this map is induced by $f$. In other words, the restriction of $f$ to $\Imr f^{n}$ is an isomorphism for all $n \geq p$. This implies that $f$ is an isomorphism whenever it is an epimorphism. Since the images of powers of $f$ form a descending chain, the restriction of any power of $f$ to $\Imr f^{n}$ is also an isomorphism for all $n \geq p$.  In particular, this is true for $f^{n}$. On the other hand, the restriction of $f^{n}$ to $\Ker f^{n}$ is zero. It now follows immediately that $\Ker f^{n} \cap \Imr f^{n}  \simeq 0$.
\medskip

(2) Since $A$ has the DCC, the inclusion $\Imr f^{n+1} \to \Imr f^{n}$ 
is an isomorphism for all $n \geq p$ for some~$p$. Therefore the inclusion  
$\Imr f^{2n} \to \Imr f^{n}$ is also an isomorphism. Thus in the commutative diagram 
\[
\xymatrix
	{
0  \ar[r] & K \ar[r] \ar[d] & \Imr f^{n}  \ar[r]^{f^{n}} \ar@{>->}[d]  &  \Imr f^{2n}  \ar[r]  
\ar[d]^{\simeq} & 0 \\
0  \ar[r] & \Ker f^{n}  \ar[r] &  A  \ar[r]^{f^{n}} & \Imr f^{n}  \ar[r] & 0 	
	}
\] 
with exact rows, the rightmost vertical map is an isomorphism. It now follows from Lemma~\ref{pushout} that the inclusion of $\Ker f^{n} + \Imr f^{n} $ in $A$ is an isomorphism. In particular, when $f$ is a monomorphism, it is also an epimorphism, and hence and isomorphism.
\medskip

(3) The union $\Ker f^{n} + \Imr f^{n}$ is the image of the morphism
\[ 
q: \Ker f^{n} \oplus \Imr f^{n} \to A
\]
induced by the inclusions of the summands in $A$. By (2), $q$ is an epimorphism. The kernel of $q$ is isomorphic (as an object, but not necessarily as a subobject) to $\Ker f^{n} \cap \Imr f^{n}$ \cite[p.~27]{B2}, which is isomorphic to $0$ by~(1). Since $q$ is an epimorphism and a monomorphism, it is an isomorphism. The rest has already been established.
\end{proof}

\section{Direct summands of functors}
Suppose $\mathscr{A}$ is an abelian category and $A$ an object of $\mathscr{A}$. We say that $A$ satisfies the \texttt{descending chain condition on images of nested endomorphisms} (DCC on INE) if any descending chain 
\[
A = A_{0} \supset A_{1} \supset A_{2} \supset \ldots
\]
of subobjects, where each $A_{i+1}$ is the image of some endomorphism of $A_{i}$, stabilizes (as a chain of subobjects). Apparently, any artinian object satisfies this condition.

Let $G : \mathscr{A}^{op} \times \mathscr{A} \to \mathtt{Ab}$ an additive (covariant) bifunctor with values in abelian groups. We make a further assumption that, for any object $A$, the map 
\[
(A,A) \to \mathrm{Nat}\big{(}G(A,-), G(A,-)\big{)} : f \mapsto G(f,-),
\]
sending a morphism to the corresponding natural transformation, is a surjection. 

\begin{theorem}\label{main}
Under the above assumptions, let $A$ be a noetherian object of 
$\mathscr{A}$ satisfying the DCC on INE. If $F$ is a direct summand of the functor  $G(A,-)$, then there is a subobject $B$ of 
$A$ such that $F \simeq G(B,-)$.  
 \end{theorem}
 
\begin{proof}
 Let 
\[
\xymatrix
	{
	F \ar[r]^<<<<<{\iota} & G(A,-) \ar[r]^>>>>>{\pi} & F
	}
\]
be a factorization of the identity map. Then the composition
\[
\xymatrix
	{
	G(A,-) \ar[r]^<<<<<{\pi} & F \ar[r]^>>>>>{\iota} & G(A,-)
	}
\]
is an idempotent endomorphism of $G(A,-)$ and, by the assumption on $G$, $\iota \pi  = G(f,-)$ for some endomorphism $f : A \to A$. If $f$ is an epimorphism in $\mathscr{A}$, then by Lemma~\ref{Fitting}, it is an isomorphism. In such a case, since both $\iota \pi$ and $1 = \pi \iota$ are isomorphisms, $F$ and  $G(A,-)$ become isomorphic, and we are done. Thus we may assume that $f : A \to A$ is not an epimorphism.

For an  epi-mono factorization of $f$ (in $\mathscr{A}$)
\[
\xymatrix
	{
	A \ar@{->>}[r]^{\alpha} & \Imr f \ar@{>->}[r]^{\beta} & A
	}
\]
through its image, we have $G(f,-) = G(\alpha,-) G(\beta,-)$. Postcomposing 
$\iota \pi  = G(f,-)$ with $\iota$, we have $\iota = G(f,-) \iota$. Setting $\delta := G(\beta,-) \iota$, we have a commutative diagram
\[
\xymatrix
	{
	& F \ar[dl]_{\iota}  \ar '[d] [dd]_{\delta}\ar[dr]^{\iota} &\\
	G(A,-) \ar[rr]^>>>>>>>>>>{G(f,-)} \ar[dr]_{G(\beta, -)}& & G(A,-)\\
	& G(\Imr f, -) \ar[ur]_{G(\alpha,-)}&
	}
\] 
Since $G(\alpha,-)\delta = \iota$ and $\iota$ is a split monomorphism, the same is true for $\delta$. This shows that $F$ is a direct summand of $G(A_{1},-)$, where $A_{1} :=\Imr f$ is a proper subobject of $A$ and at the same time an image of $A$ under an endomorphism. Repeating the foregoing argument with $A_{1}$ in place of $A$, etc., we have a descending chain $A \supset A_{1} \supset A_{2} \supset \ldots$ such that each term is the image of an endomorphism of the previous term and $F$ is a direct summand of each $G(A_{i},-)$. Since this chain must stabilize, the last endomorphism $A_{n} \to A_{n}$ must be an epimorphism, hence an isomorphism. But then, as we saw in the beginning of the proof, 
$F \simeq G(A_{n},-)$.
\end{proof}

Recall that $\mathscr{A}$ is called a \texttt{length} category if all of its objects have finite length.

\begin{theorem}
 Suppose $\mathscr{A}$ is an abelian length category and $G : \mathscr{A}^{op} \times \mathscr{A} \to \mathtt{Ab}$ an additive (covariant) bifunctor with values in abelian groups. Suppose that for any object $A$ any natural transformation 
 $G(A,-) \to G(A,-)$ is induced by a suitable endomorphism of $A$. If $F$ is a direct summand of $G(A,-)$, then there is a subobject $B$ of $A$ such that 
 $F \simeq G(B,-)$. \qed
\end{theorem}

\begin{remark}\label{covariant}
The above theorem and corollary remain true, with obvious modifications, for the functor $G(A,-)$ if $G$ is an additive (covariant) functor  on $\mathscr{A} \times \mathscr{A}$.
\end{remark}

\section{Direct summands of some homological functors}
In this section, we shall give some immediate applications of Theorem~\ref{main}. We begin with the covariant $\Ext$ functors. The relevant result concerning natural transformations between such functors is due to Hilton-Rees~\cite{HR} and says that such transformations are in one-to-one, arrow-reversing correspondence with projective equivalence classes of homomorphisms between the corresponding contravariant arguments. This means that, by assigning to each object the corresponding covariant $\Ext$ functor, one has a full contravariant embedding of the category modulo projectives in the abelian category of additive functors on the original category with values in abelian groups, where the morphisms between such functors are natural transformations.\footnote{Strictly speaking, for this statement one has to assume that the class of objects of the original category is a set, otherwise the resulting ``category'' need not be a category. The reader should notice, however, that in the proof of the forthcoming Theorem~\ref{HR} we do use categorical concepts. To justify this, we have two options. The first one is to switch from categories to quasicategories, 
see~\cite[3.49-3.51]{AHS}. The other option is to notice that, in the proof of the theorem, we do all arguments componentwise. Accordingly, the formal definition of a functor category is not needed. The details are left to the reader.} We shall now give a short and straightforward proof of this result.

\begin{theorem}[Hilton-Rees]\label{HR}
Assume that the abelian category $\mathscr{A}$ has enough projectives and let $A$ and $B$ be objects of $\mathscr{A}$. Then the correspondence 
 \[
 (B,A) \to \big{(}\Ext^{1}(A,-), \Ext^{1}(B,-)\big{)} : f \mapsto \Ext^{1}(f,-)
 \]
 induces an isomorphism 
 \[
 (\underline{B,A}) \to \mathrm{Next}^{1,1}(A,B)
 \]
between the group $(\underline{B,A})$ of morphisms from $B$ to $A$ modulo the morphisms factoring through projectives and the group $\mathrm{Next}^{1,1}(A,B)$ of natural transformations from $\Ext^{1}(A,-)$ to $\Ext^{1}(B,-)$.
\end{theorem}

\begin{proof}
An exact sequence
\[
\xymatrix
	{
\theta : &	 0 \ar[r] & \Omega A \ar[r] & P \ar[r] & A \ar[r] & 0
	}
\] 
with $P$ projective gives rise to an exact sequence of functors 
\[
\xymatrix
	{
	0 \ar[r] & (A,-) \ar[r] & (P,-) \ar[r] & (\Omega A, -) \ar[r] & \Ext^{1}(A,-) \ar[r] & 0.
	}
\] 
Taking functor morphisms into $\Ext^{1}(B,-)$ and applying Yoneda's lemma yields a commutative diagram with an exact top row
 \[
\xymatrix
	{
	0 \ar[r] & \mathrm{Next}^{1,1}(A,B) \ar[r] & \big{(}(\Omega A, -), \Ext^{1}(B,-)\big{)} \ar[d]^{\simeq} \ar[r] 	& \big{(}(P, -), \Ext^{1}(B,-)\big{)} \ar[d]^{\simeq}\\
	& & \Ext^{1}(B, \Omega A) \ar[r] &  \Ext^{1}(B, P)
	}
\] 
In particular, we see that $\mathrm{Next}^{1,1}(A,B)$ is a set. The horizontal morphism on the bottom is part of a long exact sequence and therefore its kernel is isomorphic to $\Cok ((B,P) \to (B,A))$, which is $(\underline{B,A})$. Now, $\Ext^{1}(f, -) \in \mathrm{Next}^{1,1}(A,B)$ is mapped by the horizontal homomorphism to the right multiplication by $\theta f$. Under the Yoneda isomorphism, this transformation is sent to $\theta f$, which is also the image of $f \in (B,A)$ in $\Ext^{1}(B, \Omega A)$. Thus, the inverse of the constructed isomorphism is precisely 
$[f] \mapsto \Ext^{1}(f,-)$.
\end{proof}

The above proof also works for natural transformations between contravariant $\Ext$-functors. In that case however, one needs to replace $\Hom$ modulo projectives by 
$\Hom$ modulo injectives, denoted by $(\overline{A,B})$. We thus have the Hilton-Rees theorem for the contravariant Ext functor.

\begin{theorem}\label{HRprime}
Assume that the abelian category $\mathscr{A}$ has enough injectives and let $A$ and $B$ be objects of $\mathscr{A}$. Then the correspondence 
 \[
 (A,B) \to \big{(}\Ext^{1}(-, A), \Ext^{1}(-, B)\big{)} : f \mapsto \Ext^{1}(-, f)
 \]
 induces an isomorphism 
 \[
 (\overline{A,B}) \to \mathrm{Next}(A,B)^{1,1}
 \]
between the group $(\overline{A, B})$ of morphisms from $A$ to $B$ modulo the morphisms factoring through injectives and the group  $\mathrm{Next}(A,B)^{1,1}$ of natural transformations from $\Ext^{1}(-, A)$ to $\Ext^{1}(-, B)$. \qed
\end{theorem}

Combining Theorem~\ref{HR} with Theorem~\ref{main}, we have
\begin{proposition}
Let $A$ be a noetherian object satisfying the DCC on INE in an abelian category 
$\mathscr{A}$ with enough projectives. If $F$ is a direct summand of the functor  $\Ext^{1}(A,-)$, then there is a subobject $B$ of $A$ such that $F \simeq \Ext^{1}(B,-)$. In particular, the result is true when $A$ is of finite length. \qed
\end{proposition}

\begin{corollary}
 Let $\Lambda$ be an artin algebra, $\Lambda$-$\mathrm{mod}$ the category of finitely generated (left) $\Lambda$-modules, and 
 $F : \Lambda$-$\mathrm{mod} \longrightarrow \mathtt{Ab}$ an additive functor. If $F$ is a direct summand of $\Ext^{1}(A, -)$, where $A \in \Lambda$-$\mathrm{mod}$, then there exists a submodule $B$ of~$A$
 such that $F \simeq \Ext^{1}(B, -)$. \qed
\end{corollary}

\begin{remark}
 As we mentioned in the introduction, the last corollary implies that the covariant Ext functors are the only injectives in the category of finitely presented functors 
$\Lambda$-$\mathrm{mod} \longrightarrow \mathtt{Ab}$ determined by short exact sequences. On the other hand, in general, these functors need not be injective in the ambient category of all finitely presented functors. It is not difficult to prove that if $\Lambda$ is (left) hereditary, then the functors $\Ext^{1}(B, -)$ are injective in the category of all finitely presented functors. However, even in that case, there must be injectives not of the form $\Ext^{1}(B, -)$. To see that, 
recall~\cite{G} that the category of finitely presented functors on an abelian category with enough projectives has enough injectives. Assuming that all injectives in f.p.$(\Lambda$-$\mathrm{mod} , \mathtt{Ab})$ are the Ext functors,
for any nonzero module $\Lambda$-module $A$ we have a monomorphism  
$(A,-) \longrightarrow \Ext^{1}(B, -)$ into a suitable injective. Evaluating it on the injective envelope of $A$, we then have a zero map with a nonzero domain, a contradiction.\footnote{The same argument shows that the category of finitely presented functors determined by short exact sequences does not have enough injectives. In fact this is always true: just use the fact that any injective is a direct summand of a suitable Ext functor~\cite{Aus66} and run the above argument.}
\end{remark}
\medskip

In the case when the endomorphism ring of $A$ is artinian as an abelian group, the assumption that $\mathscr{A}$ have enough projectives can be dropped. To this end,  recall the following result of Oort~\cite[p. 561]{O63}:

\begin{theorem}
 Let $\mathscr{A}$ be an abelian category all of whose objects are artinian. If
 \[
 \lambda : \Ext^{1} (B,-) \to \Ext^{1} (A,-)
 \]
 is a natural map and $\Hom(A,B)$ is an artinian group, then there exists a morphism 
 $\alpha \in \Hom(A,B)$ which induces $\lambda$. \qed
\end{theorem}

Combining this result with Theorem~\ref{main}, we have
\begin{corollary}
Suppose $\mathscr{A}$ is an abelian length category and and $A$ is an object whose endomorphism ring is artinian as an abelian group. If $F$ is a direct summand of the functor  $\Ext^{1}(A,-)$, then there is a subobject $B$ of $A$ such that $F \simeq \Ext^{1}(B,-)$. \qed
\end{corollary}

Combining Theorem~\ref{HRprime} with Theorem~\ref{main}, we have
\begin{proposition}
Let $A$ be an object of finite length in an abelian category $\mathscr{A}$ with enough injectives. If $F$ is a direct summand of the functor  $\Ext^{1}(-,A)$, then there is a subobject $B$ of $A$ such that $F \simeq \Ext^{1}(-,B)$.  \qed
\end{proposition}
\bigskip

We now turn our attention to the covariant $\Hom$ modulo projectives functor. By Yoneda's lemma, natural transformations from $(\underline{A^{\prime},-})$ to $(\underline{A,-})$ are in a functorial bijection with $(\underline{A,A^{\prime}})$. Thus, Theorem~\ref{main} yields

\begin{proposition}
Let $A$ be an object of finite length in an abelian category $\mathscr{A}$ with enough projectives. If $F$ is a direct summand of the functor  $(\underline{A,-})$, then there is a subobject $B$ of $A$ such that $F \simeq (\underline{B,-})$.  \qed
\end{proposition}
\bigskip

Next, we look at direct summands of the functor $\Tor_{1}(A,-)$. For that, we fix a ring $\Lambda$ and view $\Tor_{1}^{\Lambda}(-,-)$ as a bifunctor  f.p.\,($\mathrm{mod}$-$\Lambda$) $\times$ $\Lambda$-$\mathrm{mod}$ $\to \mathtt{Ab}$, whose first argument is taken from the category of finitely presented right $\Lambda$-modules. It is well-known (and is not difficult to prove) that, for any finitely presented right $\Lambda$-module $A$, we have a functor isomorphism $\Tor_{1}(A,-) \simeq (\underline{\Tr A, -})$, where $\Tr A$ denotes the transpose of $A$. If $A^{\prime}$ is another finitely presented right $\Lambda$-module, then, by Yoneda's lemma applied to the category of modules modulo projectives, 
\[
\mathrm{Nat}\big{(}\Tor_{1}(A,-), \Tor_{1}(A^{\prime},-)\big{)} 
\simeq \big{(}(\underline{\Tr A, -}), (\underline{\Tr A^{\prime}, -})\big{)}
\simeq (\underline{\Tr A^{\prime}, \Tr A}),
\]
But the transpose, viewed as a functor on the category of finitely presented modules modulo projectives, is a duality, and therefore the latter is isomorphic to 
$(\underline{A, A^{\prime}})$. We now have, in view of Remark~\ref{covariant}, the following 
\begin{proposition}
Let $A$ be a finitely presented right $\Lambda$-module of finite length. If~$F$ is a direct summand of the functor  $\Tor_{1}(A,-)$, then there is a finitely presented submodule 
$B$ of $A$ (automatically of finite length) such that $F \simeq \Tor_{1}(B,-)$.
\end{proposition}

\begin{proof}
The proof of Theorem~\ref{main} produces a module $B = A_{n}$, where 
\[
A = A_{0} \supset A_{1} \supset \ldots \supset A_{n} = B
\] 
is a sequence of submodules of $A$ such that each $A_{i+1}$  is a homomorphic image of $A_{i}$. We cannot immediately use this theorem, because the category of finitely presented modules does not, in general, have kernels and is not therefore abelian. But the proof would still work if we show that each $A_{i}$ is finitely presented. Since for each $i$ the kernel of the epimorphism $A_{i} \to A_{i+1}$, being a submodule of a module of finite length, is finitely generated, $A_{i+1}$ is finitely presented whenever $A_{i}$ is. But $A = A_{0}$ is finitely presented, and an induction argument finishes the proof.
\end{proof}
\begin{remark}
 The above result is also true, with obvious modifications, for left finitely presented modules.
\end{remark}

\end{document}